\documentclass[10pt]{amsart}

\usepackage{xcolor}

\usepackage[latin2]{inputenc}
\usepackage{amsmath}
\usepackage{graphicx}
\usepackage{amssymb}
\usepackage{esint}
\usepackage{color}
\usepackage{amsthm}
\usepackage{epsfig}
\usepackage{tikz}
\usepackage[english]{babel}

\allowdisplaybreaks[2]

\newtheorem{theorem}{Theorem}

\newtheorem{lemma}[theorem]{Lemma}

\newtheorem{remark}[theorem]{Remark}

\newtheorem*{theorem*}{Theorem}

\def\XXint#1#2#3{{\setbox0=\hbox{$#1{#2#3}{\int}$ }
\vcenter{\hbox{$#2#3$ }}\kern-.6\wd0}}

\definecolor{Yellow}{rgb}{0.95,0.9,0.0} 
\definecolor{Red}{rgb}{0.8,0.1,0.1}
\definecolor{Green}{rgb}{0.1,0.65,0.2}
\definecolor{Blue}{rgb}{0.1,0.1,0.8}
\definecolor{Purple}{rgb}{0.7,0.1,0.7}
\definecolor{Grey}{rgb}{0.6,0.6,0.6}

\newcommand{\supp}{\operatorname{supp}}

\newcommand{\dist}{\operatorname{dist}}

\newcommand{\sigdist}{\operatorname{dist}^{\pm}} 
\newcommand{\R}{\mathbb{R}}
\newcommand{\Rd}{\mathbb{R}^d}

\newcommand{\domain}{{\mathbb{R}^d}} 

\newcommand{\dx}{\,\mathrm{d}x}
\newcommand{\dy}{\,\mathrm{d}y}
\newcommand{\dS}{\,\mathrm{d}S}
\newcommand{\dt}{\,\mathrm{d}t}
\newcommand{\ds}{\,\mathrm{d}s}

\newcommand{\eps}{\varepsilon}
\renewcommand{\vec}[1]{{\operatorname{#1}}}
\newcommand{\sdist}{{\operatorname{dist}^\pm}}

\begin{document}

\title[Sharp-interface limit for the Allen-Cahn equation]{Convergence rates of the Allen-Cahn equation to mean curvature flow: A short proof based on relative entropies}

\author{Julian Fischer}
\author{Tim Laux}
\author{Theresa M.\ Simon}

\begin{abstract}
We give a short and self-contained proof for rates of convergence of the Allen-Cahn equation towards mean curvature flow, assuming that a classical (smooth) solution to the latter exists and starting from well-prepared initial data. Our approach is based on a relative entropy technique. In particular, it does not require a stability analysis for the linearized Allen-Cahn operator. As our analysis also does not rely on the comparison principle, we expect it to be applicable to more complex equations and systems.\\
    
\medskip


\medskip

\end{abstract}

\maketitle

\section{Introduction}

The Allen-Cahn equation
\begin{align}
\label{AllenCahn}
\frac{d}{dt} u_\eps = \Delta u_\eps - \frac{1}{\eps^2} W'(u_\eps)
\end{align}
-- with a suitable double-well potential $W$ like for instance $W(s)=c\,(1-s^2)^2$, $c>0$ -- is the most natural diffuse-interface approximation for (two-phase) mean curvature flow:
It is well-known that in the limit of vanishing interface width $\varepsilon\rightarrow 0$, the solutions $u_\eps$ to the Allen-Cahn equation \eqref{AllenCahn} converge to a characteristic function $\chi:\Rd\times [0,T] \rightarrow \{-1,1\}$ whose interface evolves by motion by mean curvature. For a proof of this fact in the framework of Brakke solutions to mean curvature flow, we refer to \cite{IlmanenConvergenceOfAllenCahn}, while for the convergence towards the viscosity solution of the level-set formulation under the assumption of non-fattening we refer to \cite{EvansSonerSouganidis}.
Provided that the total energy converges in the limit $\varepsilon\rightarrow 0$, one may prove that the limit is a distributional solution \cite{LauxSimon}.
For a general compactness statement using the gradient-flow structure of \eqref{AllenCahn} and the identification of the limit in the radially symmetric case, we refer the reader to \cite{BronsardKohn}.
Under the assumption of the existence of a smooth limiting evolution, rates of convergence may be derived based on a strategy of matched asymptotic expansions and the stability of the linearized Allen-Cahn operator \cite{Chen,DeMottoniSchatzman}.

The Allen-Cahn equation corresponds to the $L^2$ gradient flow of the Ginzburg-Landau energy functional
\begin{align}
\label{Energy}
E_\eps[v]:=\int_\domain \frac{\eps}{2} |\nabla v|^2 + \frac{1}{\eps} W(v) \dx.
\end{align}
Solutions to the Allen-Cahn equation \eqref{AllenCahn} satisfy the energy dissipation estimate
\begin{align}
\label{EnergyDissipation}
\frac{d}{dt} \int_\domain \frac{\eps}{2} |\nabla u_\eps|^2 + \frac{1}{\eps} W(u_\eps) \dx =
-\int_\domain \frac{1}{\eps} \bigg|\eps \Delta u_\eps - \frac{1}{\eps} W'(u_\eps)\bigg|^2 \dx.
\end{align}
In the present work, we pursue a strategy of deriving a quantitative convergence result in the sharp-interface limit $\varepsilon \rightarrow 0$ based purely on the energy dissipation structure.
In particular, we give a short proof for the following quantitative convergence of solutions of the Allen-Cahn equation towards a smooth solution of mean curvature flow.
\begin{theorem}
\label{Theorem1}
Let $d\in \mathbb{N}$. Let $I(t)\subset \mathbb{R}^d$, $t\in [0,T]$, be a compact interface $I(t)=\partial\Omega(t)$ evolving smoothly by mean curvature, and let $\chi:\smash{\mathbb{R}^d} \times [0,T]\rightarrow \{-1,1\}$ be the corresponding phase indicator function
\begin{align*}
\chi(x,t):=
\begin{cases}
1&\text{if }x\in \Omega(t),
\\
-1&\text{if }x\notin \Omega(t).
\end{cases}
\end{align*}
Let $W$ be a standard double-well potential as described below and denote by $\theta$ the corresponding one-dimensional interface profile.
Let $u_\eps$ be the solution to the Allen-Cahn equation \eqref{AllenCahn} with initial data given by $u_\eps(x,0)=\theta(\eps^{-1}\sdist(x,I(0)))$, where $\theta$ is the equilibrium profile defined in \eqref{optimalprofile} and $\sdist(x,I(0))$ is the signed distance function to $I(0)$ with the convention $\sdist(x,I(0))>0$ for $x\in \Omega(0)$. Define $\psi_\eps(x,t):=\int_0^{u_\eps(x,t)} \sqrt{2W(s)} \ds$. Then the error estimate
\begin{align}\label{error_estimate}
\sup_{t\in [0,T]} ||\psi_\eps(\cdot,t)-\chi(\cdot,t)||_{L^1(\mathbb{R}^d)} \leq C\big(d,T,(I(t))_{t\in [0,T]}\big) \, \varepsilon
\end{align}
holds.
\end{theorem}

\begin{remark}
	Our arguments also show that the estimate \eqref{error_estimate} holds for a larger class of solutions to the Allen-Cahn equation \eqref{AllenCahn}: 
	We only require solutions $u_\eps$ whose initial data satisfies $u_\eps(\cdot, 0) \in [-1,1]$ and whose initial relative entropy, defined below in \eqref{DefinitionRelativeEntropy}, is comparable to that of the optimal transition profile of Theorem \ref{Theorem1}, i.e., we have $E[u_\eps(\cdot,0)|I(0)] \leq C \eps^2$.
\end{remark}

We note that this error estimate is of optimal order, as $\varepsilon$ is the typical width of the diffuse interface in the Allen-Cahn approximation (i.\,e.\ the typical width of the region in which the function $\psi_\eps$ takes values in the range $[-1+\delta,1-\delta]$ for any fixed $\delta>0$).

The assumptions required for the double-well potential $W$ are standard: We require $W$ to satisfy $W(1)=W(-1)=0$ and $W(s)\geq c\min\{|s-1|^2,|s+1|^2\}$; furthermore, we require $W$ to be twice continuously differentiable, symmetric around the origin, and subject to the normalization $\smash{\int_{-1}^1 \sqrt{2W(s)} \ds}=2$. The simplest example is the normalized standard double-well potential $W(s):= \frac98 (1-s^2)^2 $. Under these assumptions, we may define the one-dimensional equilibrium profile $\theta:\mathbb{R}\rightarrow \mathbb{R}$ to be the unique odd solution of the ODE
\begin{align}\label{optimalprofile}
	\begin{cases}
		\theta'(s)&=\sqrt{2W(\theta(s))},\\
		\theta(\pm\infty)& =\pm 1;
	\end{cases}
\end{align} 
the profile $\theta$ then approaches its boundary values $\pm 1$ at $\pm \infty$ with an exponential rate, see \cite{Sternberg}.

As our quantitative convergence analysis does not rely on the comparison principle, it may be applicable to more complex models, such as systems of Navier-Stokes-Allen-Cahn type \cite{AbelsLiu}; note that a weak-strong uniqueness theorem for the two-fluid free boundary problem for the Navier-Stokes equation (i.\,e.\ the corresponding sharp-interface model) has already been obtained in \cite{FischerHensel}.
We note that a relative entropy concept related to the one in \cite{FischerHensel} had already been employed by Jerrard and Smets \cite{JerrardSmets} to deduce weak-strong uniqueness of solutions to binormal curvature flow.
In the forthcoming work \cite{FischerHenselLauxSimon}, we employ an energy-based strategy to deduce a weak-strong uniqueness theorem for multiphase mean curvature flow.

\section{Definition of the Relative Entropy and Gronwall Estimate}

\subsection{Extending the unit normal vector field of the surface evolving by mean curvature}

Let $I=I(t)$ be a surface that evolves smoothly by motion by mean curvature. Let $P_{I(t)} : \R^d \to I(t)$ be the nearest point projection to $I(t)$ and fix $r_c>0$ small enough depending on $(I(t))_{t\in [0,T]}$ so that for all $t\in [0,T]$ the map $P_{I(t)}$ is smooth in a tubular neighborhood of $I(t)$ of width $r_c$; for example one may take the minimum over $t\in[0,T]$ of the normal injectivity radius of $I_{t}$. 
For each $t\in [0,T]$, we extend the inner unit normal $\vec{n}_I$ of the surface $I(t)$ to a vector field on $\domain$ by defining
\begin{align}
\label{DefXi}
	\xi(x) := \eta(\sigdist(x,I)) \vec{n}_I( P_I(x)),
\end{align}
where $\eta$ is a cutoff for all $s\in \R$ satisfying $\eta(s)\geq 0$ and
\begin{subequations}
\begin{align}
	\eta(0) & = 1,\qquad\qquad\qquad\qquad\qquad
	&\eta(s) = 0  ~ \text{ for }  |s| \geq \frac{r_c}{2},\label{cutoffrange}\\
	\eta(s)  & \leq \max\{1 - c r_c^{-2} s^2,0\},\label{Length}\\
	|\eta'(s)| & \leq C \min\{r_c^{-1},r_c^{-2} |s|\}.
\end{align}
\end{subequations}
Furthermore, we will consider a standard cut-off $\tilde \eta$ satisfying $\tilde \eta(s) = 1$ for $|s| \leq \frac{r_c}{4}$, $\tilde \eta(s) = 0$ for $|s| \leq \frac{r_2}{2}$ and $|\eta'(s)| \leq C r_c^{-1}$, in which case one may take $\eta(s) := (1- c r_c^{-2} s^2) \tilde \eta(s)$.

The extended unit normal vector field $\xi$ and mean curvature vector $\vec{H}_I(x):=\vec{H}_I(P_I x) \tilde \eta (\dist(x,I))$ then satisfy the PDEs
\begin{subequations}
\begin{align}
\label{EquationXiTransport}
\frac{d}{dt} \xi &= -(\vec{H}_I \cdot \nabla)\xi -(\nabla \vec{H}_I)^T \xi
 + O(\dist(x,I)),
\\
\label{EquationXiLength}
\frac{d}{dt} |\xi|^2 &= -(\vec{H}_I \cdot \nabla)|\xi|^2
 + O(\dist^2(x,I)),
\end{align}
and
\begin{align}
\label{EquationXiCurvature}
-\nabla \cdot \xi = \vec{H}_I \cdot \xi + O(\dist(x,I)),
\end{align}
where the constant implicit in the $O$-notation depends on the interface $I$.
Furthermore, we have the estimate
\begin{align}
|\nabla \xi| + |\vec{H}_I| + |\nabla \vec{H}_I| \leq C(I(t)).
\label{EquationXiCurvature2}
\end{align}
\end{subequations}
To see that \eqref{EquationXiTransport} and \eqref{EquationXiLength} hold, one makes use of the formulas $\vec{n}_I(x)=\nabla \sdist(x,I)$ and $\partial_t \sdist(x,I)=-\vec{H}_I \cdot \vec{n}_I(P_Ix)$ valid in a neighborhood of $I(t)$. Formula \eqref{EquationXiCurvature} is an immediate consequence of the equality $\vec{H}_I=-(\nabla \cdot \vec{n}_I)\vec{n}_I$ valid on the interface $I(t)$ and the Lipschitz continuity of both sides of the equation.

\subsection{The relative entropy inequality}

Our argument is based on a relative entropy method. As the Modica-Mortola trick will play an important role in the definition of the relative entropy, we introduce the function
\begin{align}
\label{DefPsiEps}
\psi_\eps(x,t) := \int_0^{u_\eps(x,t)} \sqrt{2W(s)} \ds.
\end{align}
Given a smooth solution $u_\eps$ to the Allen-Cahn equation \eqref{AllenCahn} and a surface $I(t)$ which evolves smoothly by mean curvature flow, we define the relative entropy $E[u_\eps|I]$ as
\begin{align}
\label{DefinitionRelativeEntropy}
E[u_\eps|I] := \int_\domain \frac{\eps}{2}|\nabla u_\eps|^2 + \frac{1}{\eps} W(u_\eps) - \xi \cdot \nabla \psi_\eps \dx,
\end{align}
where for historic reasons we use the term ``relative entropy'' as opposed to the maybe more accurate term ``relative energy''.
Introducing the short-hand notation
\begin{subequations}
\begin{align}
\label{def:n_eps}
\vec{n}_\eps := \frac{\nabla u_\eps}{|\nabla u_\eps|}
\end{align}
(with $\vec{n}_\eps(x,t)\in\mathbb{S}^{d-1}$ arbitrary but fixed in case $|\nabla u_\eps|=0$) and writing
\begin{align*}
E[u_\eps|I] = \int_\domain \frac{\eps}{2}|\nabla u_\eps|^2 + \frac{1}{\eps} W(u_\eps) - |\nabla \psi_\eps| \dx
+\int_\domain (1-\xi\cdot \vec{n}_\eps)|\nabla \psi_\eps| \dx,
\end{align*}
we see that the relative entropy consists of two contributions: The first term
\begin{align*}
\int_\domain \frac{\eps}{2}|\nabla u_\eps|^2 + \frac{1}{\eps} W(u_\eps) - |\nabla \psi_\eps| \dx = \int_\domain \frac{1}{2}\Big|\sqrt{\eps} |\nabla u_\eps|-\frac{1}{\sqrt{\eps}}\sqrt{2W(u_\eps)}\Big|^2\dx
\end{align*}
controls the local lack of equipartition of energy between the terms $\frac{\eps}{2} |\nabla u_\eps|^2$ and $\frac{1}{\eps}W(u_\eps)$, while the second term 
\begin{align*}
	\int_\domain (1-\xi\cdot \vec{n}_\eps)|\nabla \psi_\eps| \dx\geq \frac{1}{2} \int_\domain |\vec{n}_\eps-\xi|^2 |\nabla \psi_\eps| \dx
\end{align*}
controls the local deviation of the normals $\vec{n}_\eps$ and $\vec{n}_I$. Note that the latter term also controls the distance to the interface $I(t)$ (since $|\xi|\leq \max\{1-c r_c^{-2}\dist^2(x,I),0\}$).

We furthermore introduce the notation
\begin{align}\label{def:H_eps}
\vec{H}_\eps := -\bigg(\eps \Delta u_\eps - \frac{1}{\eps} W'(u_\eps)\bigg) \frac{\nabla u_\eps}{|\nabla u_\eps|},
\end{align}
motivated by the fact that $\vec{H}_\eps$ will play the role of a curvature vector.
\end{subequations}

The key step in our analysis is the following Gronwall-type estimate for the relative entropy.
\begin{theorem}
\label{MainTheorem}
Let $I(t)$, $t\in [0,T]$, be an interface evolving smoothly by mean curvature.
Let $u_\eps$ be a solution to the Allen-Cahn equation \eqref{AllenCahn} with initial data given by $u_\eps(x,0)=\theta(\eps^{-1}\sdist(x,I(0)))$. Then for any $t\in [0,T]$ the estimate
\begin{align*}
\frac{d}{dt} E[u_\eps|I]
&+\int_\domain \frac{1}{4\eps} \big|\vec{H}_\eps - \vec{H}_I  \, \varepsilon |\nabla u_\eps|\big|^2 + \frac{1}{4\eps} \big|\vec{n}_\eps \cdot \vec{H}_\eps  - (-\nabla \cdot \xi) \sqrt{2W(u_\eps)} \big|^2 \dx
\\&
\leq C(d,(I(s))_{s\in[0,t]}) E[u_\eps|I]
\end{align*}
holds.
\end{theorem}

\subsection{Coercivity properties of the relative entropy functional}
For the proof of the Gronwall-type inequality of Theorem~\ref{MainTheorem}, we shall need the following coercivity properties of the relative entropy.
\begin{lemma}\label{lemma:control_from_entropy}
	We have the estimates
	\begin{subequations}
	\begin{align}
	\label{control_a}
		\int_\domain  \left(\sqrt{\eps} |\nabla u_\eps | - \frac{1}{\sqrt{\eps}}\sqrt{2W(u_\eps)} \right)^2 \dx & \leq  2 E[u_\eps | I],\\
	\label{control_b}
		\int_\domain  | \vec{n}_\eps- \xi |^2 |\nabla \psi_\eps| \dx & \leq 2 E[u_\eps | I],\\
	\label{control_c}
		\int_\domain | \vec{n}_\eps - \xi |^2 \eps |\nabla u_\eps|^2 \dx & \leq 12 E[u_\eps | I],
		\\
	\label{control_d}
		\int_\domain \min\{\dist^2(x,I),1\} \bigg(\frac{\eps}{2} |\nabla u_\eps|^2 + \frac{1}{\eps}W(u_\eps)\bigg)  \dx \leq & C(I) E[u_\eps|I].
	\end{align}
	\end{subequations}
\end{lemma}

\begin{proof}
We complete the square to get
\begin{align*}
  \begin{split}	
	E[u_\eps| I] 
	& = \int_\domain \frac{1}{2} \left(\sqrt{\eps} |\nabla u_\eps | - \frac{1}{\sqrt{\eps}}\sqrt{2W(u_\eps)} \right)^2 +(1- \xi \cdot \vec{n}_\eps ) |\nabla \psi_\eps| \dx.
  \end{split}
\end{align*}
In particular, we directly obtain \eqref{control_a} and \eqref{control_b} by $|\xi|\leq 1$.
By the property \eqref{Length} of the cutoff $\eta$ (and hence $1-\xi\cdot\vec{n}_\eps \geq \min\{c r_c^{-2}\dist^2(x,I),1\}$), we deduce \eqref{control_d} with $|\nabla \psi_\eps|$ instead of the energy density, which we may replace upon using \eqref{control_a}.

Employing Young's inequality in the form of
\begin{align}\nonumber
	\eps |\nabla u_\eps |^2 &= |\nabla \psi_\eps | + \sqrt{\eps} |\nabla u_\eps | \left(\sqrt{\eps} |\nabla u_\eps | - \frac{1}{\sqrt{\eps}} \sqrt{2W(u_\eps)} \right)
	\\&
	\label{replace_nabla_u_by_nabla_psi}
	\leq |\nabla \psi_\eps| + \frac{1}{2} \eps |\nabla u_\eps|^2 + \frac{1}{2} \bigg(\sqrt{\eps} |\nabla u_\eps | - \frac{1}{\sqrt{\eps}} \sqrt{2W(u_\eps)}\bigg)^2,
\end{align}
absorption and $|\vec{n}_\eps - \xi |\leq 2$ yield
\begin{align*}
&\int_\domain |\vec{n}_\eps - \xi |^2 \eps |\nabla u_\eps|^2 \, dx
\\&
\leq 2 \int_\domain |\vec{n}_\eps - \xi |^2 |\nabla \psi_\eps| \dx
+ 4 \int_\domain \bigg(\sqrt{\eps} |\nabla u_\eps | - \frac{1}{\sqrt{\eps}} \sqrt{2W(u_\eps)}\bigg)^2 \dx.
\end{align*}
By \eqref{control_a} and \eqref{control_b}, this shows \eqref{control_c}.
\end{proof}

\subsection{Time evolution of the relative entropy functional}
The main step in the proof of Theorem~\ref{MainTheorem} is the derivation of the following formula; by estimating the right-hand side using the abovementioned coercivity properties and equations \eqref{EquationXiTransport}--\eqref{EquationXiCurvature}, we will derive the Gronwall-type inequality of Theorem~\ref{MainTheorem}.
\begin{lemma}
Let $u_\eps$ be a solution to the Allen-Cahn equation \eqref{AllenCahn} and let $I=I(t)$ be a smooth solution to mean curvature flow. Let $\xi$ be as defined in \eqref{DefXi}. The time evolution of the relative entropy is then given by
\begin{align}
\nonumber
\frac{d}{dt} E[u_\eps|I]
&=
-\int_\domain \frac{1}{2\eps} \big|\vec{H}_\eps - \vec{H}_I  \, \varepsilon |\nabla u_\eps|\big|^2 + \frac{1}{2\eps} \big|\vec{n}_\eps \cdot \vec{H}_\eps  - (-\nabla \cdot \xi) \sqrt{2W(u_\eps)} \big|^2 \dx
\\&~~~~
\nonumber
+\int_\domain |\vec{H}_I|^2 \frac{\eps}{2} |\nabla u_\eps|^2 + |\nabla \cdot \xi|^2 \frac{1}{\varepsilon} W(u_\eps) + \vec{H_I}\cdot \vec{n}_\eps (\nabla \cdot \xi) |\nabla \psi_\eps| \dx
\\&~~~~
\nonumber
+\int_\domain \nabla \cdot \vec{H}_I \bigg(\frac{\eps}{2}|\nabla u_\eps|^2
+ \frac{1}{\eps} W(u_\eps) - |\nabla \psi_\eps| \bigg) \dx
\\&~~~~
\nonumber
- \int_\domain \nabla \vec{H}_I : \vec{n}_\eps\otimes \vec{n}_\eps (\eps |\nabla u_\eps|^2 - |\nabla \psi_\eps|) \dx
\\&~~~~
\label{time_derivative_entropy}
-\int_\domain \nabla \vec{H}_I : (\vec{n}_\eps-\xi) \otimes (\vec{n}_\eps -\xi) |\nabla \psi_\eps| \dx
\\&~~~~
\nonumber
+\int_\domain \nabla \cdot \vec{H}_I (1-\xi\cdot \vec{n}_\eps) |\nabla \psi_\eps| \dx
\\&~~~~
\nonumber
-\int_\domain |\nabla \psi_\eps| (\vec{n}_\eps -\xi) \cdot \bigg(\frac{d}{dt} \xi +  (\vec{H}_I \cdot \nabla) \xi+(\nabla \vec{H}_I)^T \xi\bigg) \dx
\\&~~~~
\nonumber
-\int_\domain |\nabla \psi_\eps| \xi \cdot \bigg(\frac{d}{dt} \xi +  (\vec{H}_I \cdot \nabla) \xi \bigg) \dx.
\end{align}
\end{lemma}
\begin{proof}
By direct computation, we obtain
\begin{align*}
\frac{d}{dt} E[u_\eps|I] & =\frac{d}{dt} \int_\domain \frac{\eps}{2} |\nabla u_\eps|^2 + \frac{1}{\eps} W(u_\eps) - \xi \cdot \nabla \psi_\eps \dx
\\&
\stackrel{\eqref{EnergyDissipation},\eqref{AllenCahn}}{=}
-\int_\domain \frac{1}{\eps} \bigg|\eps \Delta u_\eps - \frac{1}{\eps} W'(u_\eps)\bigg|^2 \dx
\\&~~~~~~~~
-\int_\domain \nabla \psi_\eps \cdot \frac{d}{dt} \xi \dx
+ \int_\domain \sqrt{2W(u_\eps)} \bigg(\Delta u_\eps - \frac{1}{\eps^2} W'(u_\eps)\bigg) \nabla \cdot \xi \dx.
\end{align*}
With the definitions \eqref{def:n_eps} and \eqref{def:H_eps},
we obtain
\begin{align*}
\frac{d}{dt} E[u_\eps|I]
&=\int_\domain -\frac{1}{\eps} |\vec{H}_\eps|^2  + \vec{n}_\eps \cdot \vec{H}_\eps  \, (-\nabla \cdot \xi) ~ \frac{1}{\eps} \sqrt{2W(u_\eps)}
\dx
\\&~~~~
+\int_\domain \nabla \vec{H}_I : \xi \otimes \vec{n}_\eps |\nabla \psi_\eps| \dx
\\&~~~~
+\int_\domain (\vec{H}_I \cdot \nabla) \xi \, \cdot \nabla \psi_\eps \dx
\\&~~~~
-\int_\domain \nabla \psi_\eps \cdot \bigg(\frac{d}{dt} \xi +  (\vec{H}_I \cdot \nabla) \xi+(\nabla \vec{H}_I)^T \xi\bigg) \dx.
\end{align*}
We exploit the symmetry of the Hessian $\nabla^2 \psi_\eps$
\begin{align*}
&\int_\domain (\vec{H}_I \cdot \nabla) \xi \, \cdot \nabla \psi_\eps \dx
\\&
= -\int_\domain \vec{H}_I \otimes \xi : \nabla^2 \psi_\eps \dx -\int_\domain \nabla \cdot \vec{H}_I ~ \xi \cdot \nabla \psi_\eps \dx
\\&
=\int_\domain (\xi\cdot \nabla) \vec{H}_I \cdot \nabla \psi_\eps \dx+ \int_\domain (\nabla \cdot \xi ~ \vec{H}_I - \nabla \cdot \vec{H}_I ~\xi) \cdot \nabla \psi_\eps \dx,
\end{align*}
which yields 
\begin{align*}
\frac{d}{dt} E[u_\eps|I]
&=\int_\domain -\frac{1}{\eps} |\vec{H}_\eps|^2  + \vec{n}_\eps \cdot \vec{H}_\eps  \, (-\nabla \cdot \xi) ~ \frac{1}{\eps} \sqrt{2W(u_\eps)}
\dx
\\&~~~~
+\int_\domain \nabla \vec{H}_I : \xi \otimes \vec{n}_\eps |\nabla \psi_\eps| \dx
\\&~~~~
+\int_\domain (\xi\cdot \nabla) \vec{H}_I \cdot \vec{n}_\eps |\nabla \psi_\eps| \dx
\\&~~~~
+\int_\domain (\nabla \cdot \xi ~ \vec{H}_I - \nabla \cdot \vec{H}_I ~\xi) \cdot \nabla \psi_\eps \dx
\\&~~~~
-\int_\domain \nabla \psi_\eps \cdot \bigg(\frac{d}{dt} \xi +  (\vec{H}_I \cdot \nabla) \xi+(\nabla \vec{H}_I)^T \xi\bigg) \dx.
\end{align*}
Together with $\xi \otimes \vec{n}_\eps + \vec{n}_\eps \otimes \xi  = - (\vec{n}_\eps -\xi) \otimes (\vec{n}_\eps - \xi) + \vec{n}_\eps \otimes \vec{n}_\eps + \xi \otimes \xi$ the computation \eqref{Auxiliary} below then implies 
\begin{align*}
\frac{d}{dt} E[u_\eps|I]
&=\int_\domain -\frac{1}{\eps} |\vec{H}_\eps|^2 + \vec{H}_\eps \cdot \vec{H}_I \, |\nabla u_\eps| + \vec{n}_\eps \cdot \vec{H}_\eps  \, (-\nabla \cdot \xi) ~ \frac{1}{\eps} \sqrt{2W(u_\eps)}
\dx
\\&~~~~
+\int_\domain \nabla \cdot \vec{H}_I |\nabla \psi_\eps| \dx
\\&~~~~
+\int_\domain \nabla \cdot \vec{H}_I \bigg(\frac{\eps}{2}|\nabla u_\eps|^2
+ \frac{1}{\eps} W(u_\eps) - |\nabla \psi_\eps| \bigg) \dx
\\&~~~~
- \int_\domain \nabla \vec{H}_I : \vec{n}_\eps\otimes \vec{n}_\eps (\eps |\nabla u_\eps|^2 - |\nabla \psi_\eps|) \dx
\\&~~~~
-\int_\domain \nabla \vec{H}_I : (\vec{n}_\eps-\xi) \otimes (\vec{n}_\eps -\xi) |\nabla \psi_\eps| \dx
\\&~~~~
+\int_\domain (\xi\cdot \nabla) \vec{H}_I \cdot \xi  |\nabla \psi_\eps| \dx
\\&~~~~
+\int_\domain (\nabla \cdot \xi ~ \vec{H}_I - \nabla \cdot \vec{H}_I ~\xi) \cdot \nabla \psi_\eps \dx
\\&~~~~
-\int_\domain \nabla \psi_\eps \cdot \bigg(\frac{d}{dt} \xi +  (\vec{H}_I \cdot \nabla) \xi+(\nabla \vec{H}_I)^T \xi\bigg) \dx.
\end{align*}
Completing the squares and adding zero, we obtain \eqref{time_derivative_entropy}.
\end{proof}

\subsection{Auxiliary computation}

In the above computation, we have made use of the formula
\begin{align}
\label{Auxiliary}
& \quad \int_\domain \nabla \vec{H}_I : \vec{n}_\eps \otimes \vec{n}_\eps |\nabla \psi_\eps |\dx
\\
\nonumber
&=\int_\domain  \vec{H}_\eps \cdot \vec{H}_I |\nabla u_\eps| \dx
+\int_\domain \nabla \cdot \vec{H}_I \bigg(\frac{\eps}{2}|\nabla u_\eps|^2
+ \frac{1}{\eps} W(u_\eps) \bigg) \dx
\\&\quad
\nonumber
- \int_\domain \nabla \vec{H}_I : \vec{n}_\eps\otimes \vec{n}_\eps (\eps |\nabla u_\eps|^2 - |\nabla \psi_\eps|) \dx.
\end{align}
Indeed, due to definition \eqref{def:H_eps} we have
\begin{align*}
-\int_\domain  \vec{H}_\eps \cdot \vec{H}_I |\nabla u_\eps| \dx =  \int_\domain \left(\eps \Delta u_\eps - \frac{W'(u_\eps)}{\eps}\right) \vec{H}_I \cdot \nabla u_\eps \dx.
\end{align*}
Using the identity $\sum_{i=1}^d \partial_i \partial_i u_\eps \partial_j u_\eps = \sum_{i=1}^d \left( \partial_i(\partial_i u_\eps \partial_j u_\eps)\right) - \frac{1}{2}\partial_j |\nabla u_\eps |^2$ we calculate
\begin{align*}
	& \quad\int_\domain \left(\eps \Delta u_\eps - \frac{W'(u_\eps)}{\eps}\right) \vec{H}_I \cdot \nabla u_\eps \dx \\
	& = \int_\domain  \sum_{i,j=1}^d \left( \eps \partial_i \partial_i u_\eps \partial_j u_\eps \vec{H}_{I,j}\right) - \frac{1}{\eps} \vec{H}_I \cdot \nabla \left( W( u_\eps)\right)   \dx \\
	& = \int_\domain  \sum_{i,j=1}^d \left( - \eps \partial_i \vec{H}_{I,j}\partial_i u_\eps \partial_j u_\eps \right) +\nabla \cdot \vec{H}_I  \left( \frac{\eps}{2} |\nabla u_\eps|^2 + \frac{W(u_\eps)}{\eps}\right)\dx.
\end{align*}
Recalling the abbreviation $\vec{n}_\eps = \frac{\nabla u_\eps}{|\nabla u_\eps|}$ we get
\begin{align}\label{partial_integration_cross_term}
  \begin{split}
	& \quad - \int_\domain  \vec{H}_\eps \cdot \vec{H}_I |\nabla u_\eps| \dx \\
	& = \int_\domain \nabla \cdot \vec{H}_I  \left( \frac{\eps}{2} |\nabla u_\eps|^2 + \frac{W(u_\eps)}{\eps}\right) - \nabla \vec{H_I} :\left( \vec{n}_\eps \otimes \vec{n}_\eps \right) \eps |\nabla u_\eps |^2 \dx.
  \end{split}
\end{align}
With the goal of replacing the expressions $\frac{\eps}{2} |\nabla u_\eps|^2 + \frac{W(u_\eps)}{\eps}$ and $\eps |\nabla u_\eps|^2$ by $|\nabla \psi_\eps |$ we rewrite the identity \eqref{partial_integration_cross_term} as \eqref{Auxiliary}.

\subsection{Derivation of the Gronwall inequality}

\begin{proof}[Proof of Theorem~\ref{MainTheorem}]
Using the estimates of Lemma~\ref{lemma:control_from_entropy} we can control the terms on the right-hand side of the identity \eqref{time_derivative_entropy}. Using \eqref{EquationXiTransport}, \eqref{EquationXiLength}, and the bound $||\nabla H_I||_{L^\infty}\leq C(I(t))$, the last four lines of \eqref{time_derivative_entropy} may be estimated by
\begin{align*}
C(I(t)) \int_\domain \min\{\dist^2(x,I),1\}  |\nabla \psi_\eps| + |\vec{n}_\eps-\xi|^2 |\nabla \psi_\eps| + (1-\vec{n}_\eps\cdot\xi) |\nabla \psi_\eps| \dx,
\end{align*}
which by \eqref{control_b} and \eqref{control_d} is bounded by $C(I(t)) E[u_\eps|I]$.

The third line on the right-hand side of \eqref{time_derivative_entropy} can be estimated as
\begin{align}
	\int_\domain \left| \nabla \cdot \vec{H}_I \bigg(\frac{\eps}{2}|\nabla u_\eps|^2
+ \frac{1}{2\eps} W(u_\eps) - |\nabla \psi_\eps| \bigg) \right| \dx \leq \| \nabla \cdot \vec{H}_I \|_\infty E[u_\eps | I ].
\end{align}
Thus, it only remains to estimate the second and the fourth term on the right-hand side of \eqref{time_derivative_entropy}.

Concerning the fourth term, we use the fact that $(\xi\cdot \nabla) \vec{H}_I \equiv 0$ holds in a neighborhood of $I(t)$, Young's inequality, and \eqref{DefPsiEps} to deduce 
\begin{align*}
& \int_\domain \left| \nabla \vec{H}_I : \vec{n}_\eps\otimes \vec{n}_\eps (\eps |\nabla u_\eps|^2 - |\nabla \psi_\eps|) \right|\dx
\\
& = \int_\domain \left| \nabla \vec{H}_I : \vec{n}_\eps \otimes (\vec{n}_\eps - \xi ) \left(\eps |\nabla u_\eps |^2 - |\nabla \psi_\eps| \right)\right| \dx
\\&~~~~
+C\int_\domain \min\{\dist^2(x,I),1\} \big(\eps |\nabla u_\eps |^2 +|\nabla \psi_\eps| \big) \dx
\\
	&  \leq \|\nabla \vec{H}_I \|_\infty \left( \int_\domain |\vec{n}_\eps -\xi |^2 \eps |\nabla u_\eps|^2\, dx \right)^\frac{1}{2} \left( \int_\domain \left(\sqrt{\eps}|\nabla u_\eps| - \frac{1}{\sqrt{\eps}} \sqrt{2W(u_\eps)} \right)^2  \dx \right)^\frac{1}{2}
\\&~~~~
+C\int_\domain \min\{\dist^2(x,I),1\} \big(\eps |\nabla u_\eps |^2 +|\nabla \psi_\eps|\big) \dx.
\end{align*}
Consequently, Lemma \ref{lemma:control_from_entropy} implies that the fourth line on the right-hand side of \eqref{time_derivative_entropy} is bounded by $C E[u_\eps | I]$.

It only remains to bound the term in the second line of the right-hand side of \eqref{time_derivative_entropy}. To this aim, we complete the square and estimate
\begin{align*}
&\int_\domain |\vec{H}_I|^2 \frac{\eps}{2} |\nabla u_\eps|^2 + |\nabla \cdot \xi|^2 \frac{1}{\varepsilon} W(u_\eps) +\vec{H}_I \cdot \vec{n}_\eps \nabla \cdot \xi  |\nabla \psi_\eps| \dx
\\&
=\int_\domain
\frac{1}{2}\bigg|\sqrt{\eps} |\nabla u_\eps| \vec{H}_I+\frac{1}{\sqrt{\eps}} \nabla \cdot \xi \sqrt{2W(u_\eps)} \vec{n}_\eps \bigg|^2
\dx
\\&
\leq
\frac{3}{2}
\int_\domain
\bigg|(\nabla \cdot \xi) \vec{n}_\eps \Big(\sqrt{\eps} |\nabla u_\eps| - \frac{1}{\sqrt{\eps}} \sqrt{2W(u_\eps)}\Big) \bigg|^2
\dx
\\&~~~~
+\frac{3}{2}
\int_\domain
\bigg|(\nabla \cdot \xi)(\xi-\vec{n}_\eps) \sqrt{\eps} |\nabla u_\eps| \bigg|^2
\dx
\\&~~~~
+\frac{3}{2}
\int_\domain
\bigg|\big(\vec{H}_I + (\nabla \cdot \xi) \xi\big) \sqrt{\eps} |\nabla u_\eps| \bigg|^2
\dx.
\end{align*}
Inserting the estimates \eqref{EquationXiCurvature} and \eqref{EquationXiCurvature2} and using the fact that $H_I= (H_I\cdot \xi) \xi+O(\dist(x,I))$, we obtain
\begin{align*}
&\int_\domain |\vec{H}_I|^2 \frac{\eps}{2} |\nabla u_\eps|^2 + |\nabla \cdot \xi|^2 \frac{1}{\varepsilon} W(u_\eps) +\vec{H}_I \cdot \vec{n}_\eps \nabla \cdot \xi |\nabla \psi_\eps| \dx
\\&
\leq C \int_\domain \bigg|\sqrt{\eps} |\nabla u_\eps| - \frac{1}{\sqrt{\eps}} \sqrt{2W(u_\eps)} \bigg|^2 \dx
\\&~~~
+C \int_\domain |\vec{n}_\eps-\xi|^2 \eps |\nabla u_\eps|^2+ \min\{\dist^2(x,I),1\} \eps|\nabla u_\eps|^2 \dx.
\end{align*}
By Lemma~\ref{lemma:control_from_entropy}, we see that these terms are estimated by $CE[u_\eps|I]$.
\end{proof}

\pagebreak
\section{Estimate for the Interface Error}

We now derive the interface error estimate of Theorem~\ref{Theorem1}.
\begin{proof}[Proof of Theorem~\ref{Theorem1}]
\emph{Step 1: Estimate for the relative entropy.} In view of Theorem~\ref{MainTheorem}, in order to prove
\begin{align}
\label{InterfaceErrorEstimateEnergy}
&\sup_{t\in [0,T]} E[u_\eps|I] + \int_0^T \int_\domain \frac{1}{\eps} \big|\vec{H}_\eps - \vec{H}_I  \, \varepsilon |\nabla u_\eps|\big|^2 + \frac{1}{\eps} \big|\vec{n}_\eps \cdot \vec{H}_\eps  - (-\nabla \cdot \xi) \sqrt{2W(u_\eps)} \big|^2 \dx \dt
\\&~~~~~~~~~
\nonumber
\leq C(d,T,(I(t))_{t\in [0,T]}) \eps^2
\end{align}
it only remains to show that the initial relative entropy satisfies $E[u_\eps|I](0)\leq C(d,I(0)) \eps^2$. To this end, we compute using $u_\eps(x,0)=\theta(\eps^{-1}\sdist(x,I(0)))$ and the fact that $\nabla \sdist(x,I(0))\cdot \xi = |\nabla \sdist(x,I(0))| |\xi|\geq |\xi|^2$
\begin{align*}
E[u_\eps|I](0)
\leq&
\int_\domain
\frac{|\xi|^2}{2\eps} |\theta'(\eps^{-1}\sdist(x,I(0)))|^2
+\frac{|\xi|^2}{\eps}W(\theta(\eps^{-1}\dist^\pm(x,I(0)))
\\&~~~~
-\frac{1}{\eps}
\sqrt{2W(\theta(\eps^{-1}\dist^\pm(x,I(0)))} \, \theta'(\eps^{-1}\sdist(x,I(0))) |\xi|^2
\dx
\\&
+\int_\domain \frac{1}{\eps} (1-|\xi|^2) \Big( \frac12|\theta'(\eps^{-1}\sdist(x,I(0)))|^2+W(\theta(\eps^{-1}\sdist(x,I(0))))\Big) \dx.
\end{align*}
Using the defining equation $\theta'(s)= \sqrt{2W(\theta(s))}$ as well as the fact that $|\theta'(s)|$ decays exponentially in $s$ and that $|\xi|^2\geq 1-c\dist^2(x,I)$, we deduce $E[u_\eps|\chi](0)\leq C(d,I(0)) \eps^2$.

~\\
\noindent
\emph{Step 2: Interface error estimate.}
We now perform an additional computation to obtain a more explicit control on the interface error.
We may write
\begin{align*}
\partial_t \psi_\eps = \sqrt{2W(u_\eps)} \partial_t u_\eps \stackrel{\eqref{AllenCahn},\eqref{def:H_eps}}{=} -\eps^{-1} \sqrt{2 W(u_\eps)} \vec{H}_\eps \cdot \vec{n}_\eps.
\end{align*}
We choose $\tau:\mathbb{R} \rightarrow [-1,1]$ to be a smooth monotone truncation of the identity map (with $\tau(s) \geq \min\{s,\frac{1}{2}\}$ for $s>0$, $\tau(s) \leq \max\{s,-\frac{1}{2}\}$ for $s<0$ and $\tau(s)=\operatorname{sign}(s)$ for $|s|\geq 1$). Fixing $s_0>0$ to be determined later and observing that the measure-function pairing $(\frac{d}{dt}\chi) \tau\big(\frac1{s_0}\sdist(x,I)\big)$ vanishes, we obtain
\begin{align*}
&\frac{d}{dt} \int_\domain (\chi-\psi_\eps) \tau\Big(\frac1{s_0}\sdist(x,I)\Big) \dx
\\&
=\int_\domain \eps^{-1}  \sqrt{2 W(u_\eps)} \vec{H}_\eps \cdot \vec{n}_\eps \tau\Big(\frac1{s_0}\sdist(x,I)\Big) \dx
\\&~~~~
+ \int_\domain (\chi-\psi_\eps) \frac1{s_0}\tau'\Big(\frac1{s_0}\sdist(x,I)\Big) \partial_t \sdist(x,I) \dx
\\&
=\int_\domain \eps^{-1} \sqrt{2 W(u_\eps)} \vec{H}_\eps \cdot \vec{n}_\eps \tau\Big(\frac1{s_0}\sdist(x,I)\Big) \dx
\\&~~~~
- \int_\domain (\chi-\psi_\eps) \vec{H}_I \cdot \nabla \Big(\tau\Big(\frac1{s_0}\sdist(x,I)\Big)\Big) \dx
\\&~~~~
+ \int_\domain (\chi-\psi_\eps) \frac1{s_0}\tau'\Big(\frac1{s_0}\sdist(x,I)\Big) \big(\partial_t \sdist(x,I)+\vec{H}_I\cdot \nabla \sdist(x,I)\big) \dx
\\&
=\int_\domain (\eps^{-1} \sqrt{2 W(u_\eps)} \vec{H}_\eps \cdot \vec{n}_\eps- \nabla \psi_\eps \cdot \vec{H}_I) \tau\Big(\frac1{s_0}\sdist(x,I)\Big) \dx
\\&~~~~
+ \int_\domain (\chi-\psi_\eps)  \tau\Big(\frac1{s_0}\sdist(x,I)\Big)\nabla \cdot \vec{H}_I \dx
\\&~~~~
+ \int_\domain (\chi-\psi_\eps) \frac1{s_0} \tau'\Big(\frac1{s_0}\sdist(x,I)\Big) \big(\partial_t \sdist(x,I)+H_I\cdot \nabla \sdist(x,I)\big) \dx
\end{align*}
where in the last step we have used integration by parts and $\tau(\sdist(x,I(t)) = 0$ on $\supp \nabla \chi(\,\cdot\,,t)$.

This may be rewritten using the definition of $\psi_\eps$ and $\vec{n}_\eps$ as
\begin{align*}
&\frac{d}{dt} \int_\domain (\chi-\psi_\eps) \tau\Big(\frac1{s_0}\sdist(x,I)\Big) \dx
\\&
=\int_\domain \eps^{-1} \sqrt{2 W(u_\eps)} (\vec{H}_\eps -\vec{H}_I \eps |\nabla u_\eps|)\cdot \vec{n}_\eps \tau\Big(\frac1{s_0}\sdist(x,I)\Big) \dx
\\&~~~
+\int_\domain (\chi-\psi_\eps)  \tau\Big(\frac1{s_0}\sdist(x,I)\Big) \nabla \cdot \vec{H}_I \dx
\\&~~~
+ \int_\domain (\chi-\psi_\eps) \frac1{s_0}\tau'\Big(\frac1{s_0}\sdist(x,I)\Big) \big(\partial_t \sdist(x,I)+\vec{H}_I\cdot \nabla \sdist(x,I)\big) \dx.
\end{align*}
Since $\partial_t \sdist(x,I)=-\vec{H}_I \cdot \nabla \sdist(x,I)$ holds in a neighborhood of the interface, the last integral vanishes identically if we choose $s_0>0$ sufficiently small, e.g., $s_0= \frac{r_c}{4}$. Using Cauchy-Schwarz we deduce
\begin{align*}
&\frac{d}{dt} \int_\domain (\chi-\psi_\eps) \tau\Big(\frac1{s_0}\sdist(x,I)\Big) \dx
\\&
\leq \int_\domain \eps^{-1} \big|\vec{H}_\eps - \vec{H}_I \eps |\nabla u_\eps|\big|^2 \dx
+\int_\domain \eps^{-1} 2W(u_\eps) \Big|\tau\Big(\frac1{s_0}\sdist(x,I)\Big)\Big|^2 \dx
\\&~~~
+||(\nabla \cdot \vec{H}_I)_+||_{L^\infty} \int_\domain |\psi_\eps-\chi| \Big|\tau\Big(\frac1{s_0}\sdist(x,I)\Big)\Big| \dx,
\end{align*}
where $(\nabla \cdot \vec{H}_I)_+$ denotes the positive part of $\nabla \cdot \vec{H}_I$. In order to be able to apply the Gronwall inequality, we note that $\psi_\eps \in [-1,1]$. The most natural way of ensuring this is by using the maximum principle, although also a purely energetic proof by means of the minimizing movements scheme and a truncation argument is available.
By the Gronwall inequality, \eqref{InterfaceErrorEstimateEnergy}, \eqref{control_d}, this shows that
\begin{align}
\label{GronwallBoundFull}
\sup_{t\in [0,T]} \int_\domain |\psi_\eps-\chi| \min\{\dist(x,I),1\} \dx \leq C(d,T,(I(t))_{t\in [0,T]}) \eps^2.
\end{align}

In order to pass to an unweighted norm we use the following elementary estimate for $f\in L^\infty(0,\frac{r_c}{2})$:
Applying Fubini's theorem after splitting the square $[0,\frac{r_c}{2}]^2$ into two isoceles triangles yields 
\begin{align*}
& \bigg(\int_{0}^\frac{r_c}{2} |f(y)| \dy\bigg)^2
\leq 2 \| f\|_\infty
\int_{0}^\frac{r_c}{2} |f(y)| y \dy.
\end{align*}
This allows to estimate for the $\frac{r_c}{2}$-neighborhood of $I(t)$
\begin{align*}
&\bigg(\int_{I(t)+B_\frac{r_c}{2}} |\psi_\eps(x,t)-\chi(x,t)| \dx\bigg)^2
\\&
\leq C(d,I(t))
\bigg(\int_{I(t)} \int_{0}^\frac{r_c}{2} |\psi_\eps(w+y\vec{n}_I(w),t)-\chi(w+y\vec{n}_I(w),t)| \dy
\\&~~~~~~~~~~~~~~~~~~~~~~~~~~
+ \int_{0}^\frac{r_c}{2} |\psi_\eps(w-y\vec{n}_I(w),t)-\chi(w-y\vec{n}_I(w),t)| \dy \dS(w)\bigg)^2
\\&
\leq C(d,I(t))
\int_{I(t)} \int_{-\frac{r_c}{2}}^{\frac{r_c}{2}} |\psi_\eps(w+y\vec{n}_I(w),t)-\chi(w+y\vec{n}_I(w),t)|
\\&~~~~~~~~~~~~~~~~~~~~~~~~~~~~~~~~~~~~~~
\times
\dist(w+y\vec{n}_I(w),I(t)) \dy \dS(w)
\\&
\leq C(d,I(t))
\int_{I(t)+B_{\frac{r_c}{2}}} |\psi_\eps(x,t)-\chi(x,t)| \dist(x,I) \dx 
,
\end{align*}
which in view of \eqref{GronwallBoundFull} yields Theorem~\ref{Theorem1}.
\end{proof}

\bibliographystyle{abbrv}
\bibliography{allencahn}

\end{document}